\documentclass[11pt]{article}

\usepackage[english]{babel}
\usepackage{amsmath,amsfonts,amssymb,amsthm}
\usepackage{latexsym}
\usepackage{makeidx}
\usepackage{color}
\usepackage{tikz}
\usepackage{vmargin}
\usepackage{enumerate}
\setmarginsrb{20mm}{20mm}{20mm}{20mm}%
            {5mm}{5mm}{5mm}{10mm}

\numberwithin{equation}{section}

\numberwithin{table}{section}

\usepackage{accents}
\newlength{\dhatheight}
\newcommand{\doublehat}[1]{%
    \settoheight{\dhatheight}{\ensuremath{\widehat{#1}}}%
    \addtolength{\dhatheight}{-0.35ex}%
    \widehat{\vphantom{\rule{0pt}{\dhatheight}}%
    \smash{\widehat{#1}}}}

\usepackage{scalerel,stackengine}
\stackMath
\newcommand\reallywidehat[1]{%
\savestack{\tmpbox}{\stretchto{%
  \scaleto{%
    \scalerel*[\widthof{\ensuremath{#1}}]{\kern-.6pt\bigwedge\kern-.6pt}%
    {\rule[-\textheight/2]{1ex}{\textheight}}
  }{\textheight}%
}{0.5ex}}%
\stackon[1pt]{#1}{\tmpbox}%
}
\usepackage{caption}
\usepackage[all]{xy}

\theoremstyle{plain}
\newtheorem{defn}{Definition}[section]
\newtheorem{lem}[defn]{Lemma}

\newtheorem*{thmA}{Theorem A}
\newtheorem*{thmB}{Theorem B}
\newtheorem*{thmC}{Theorem C}

\newtheorem*{bridge}{Bridge Theorem}

\newtheorem{cor}[defn]{Corollary}
\newtheorem{prop}[defn]{Proposition}

\newtheorem{fact}[defn]{Fact}

\theoremstyle{remark}
\newtheorem{ex}[defn]{Example}
\newtheorem{rem}[defn]{Remark}

\global\def\ca#1{\mathcal{#1}}

\newcommand{\K}{\mathbb K}
\newcommand{\N}{\mathbb N}
\newcommand{\Z}{\mathbb Z}

\newcommand{\B}{\mathcal{B}}
\newcommand{\BV}{\B(V)}
\newcommand{\C}{\mathcal{C}}

\newcommand{\LLC}{_{\K}\mathrm{LLC}}
\newcommand{\LC}{_{\K}\mathrm{LC}}
\newcommand{\Vect}{_{\K}\mathrm{Vect}}

\def\mod{\mathrm{Mod}}
\def\Flow{\mathrm{Flow}}

\def\f{\phi}
\def\ent{\mathrm{ent}}

\def\Hom{\mathrm{CHom}}

\def\f{\phi}
\def\ent{\mathrm{ent}}

\def\Hom{\mathrm{Hom}}

\global\def\codi#1{\mathrm{codim}\,{#1}}

\global\def\ca#1{\mathcal{#1}}

\def\rk{\mathrm{rk}}

\title{The corank of a flow over the category of linearly compact vector spaces.}

\begin{document}
\date{}
\author{Ilaria Castellano\\{\footnotesize {\tt ilaria.castellano88@gmail.com}} \\ {\footnotesize University of Southampton}\\ {\footnotesize Building 54, Salisbury Road - 50171BJ Southampton (UK)}} 
\maketitle
\begin{abstract} For a topological flow $(V,\f)$ - i.e., $V$ is a linearly compact vector space and $\f$ a continuous endomorphism of $V$ - we gain  a deep understanding of the relationship between  $(V,\f)$ and the Bernoulli shift: a topological flow $(V,\phi)$ is essentially a product of one-dimensional left Bernoulli shifts as many as $\ent^*(V,\f)$ counts. This novel comprehension brings us to introduce a notion of {\it corank} for topological flows designed for coinciding with the value of the topological entropy of $(V,\f)$.  
As an application, we provide an alternative proof of the so-called Bridge Theorem for locally linearly compact vector spaces connecting the topological entropy to the algebraic entropy by means of Lefschetz duality.
 \end{abstract}
 
 \bigskip
\noindent\emph{Key words and phrases: linearly compact vector space, algebraic entropy, topological entropy, algebraic dynamical system, dual goldie dimension, corank, Bernoulli shift.} 

\section{Introduction}
 Adler, Konheim and McAndrew \cite{AKM} introduced the notion of topological entropy $\mathrm{h}_{\textup{top}}$ for continuous self-maps of compact spaces, and they sketched a definition of the algebraic entropy $\mathrm{h}_{\textup{alg}}$ for endomorphisms of abelian groups. Later on these entropy functions have been pushed forward by many different authors, e.g., \cite{Bowen,Hood,Weiss,SZ} (see also \cite{DGB1} for a complete survey). 

 Here we investigate the topological entropy $\ent^*$ and the algebraic entropy $\ent$ introduced in \cite{CGB1,CGB2, IC} for continuous endomorphisms of  locally  linearly compact vector spaces (see \S~\ref{s:preliminaries} for all definitions and properties). In particular, the ultimate aim of this paper is to characterise the topological entropy $\ent^*$ for linearly compact vector spaces in terms of a corank function; and this aim is achieved with Theorem~C in Section~5.
 
 \medskip

Let $\LC$ denote the category of linearly compact $\K$-spaces and their continuous endomorphisms. A {\it topological flow} is a pair $(V,\phi)$ consisting of an object $V$ of $\LC$ and an endomorphism $\phi$ of $ V$.  Recall that the topological entropy $\ent^*$ for linearly compact $\K$-spaces is actually an invariant $$\ent^*\colon \Flow(\LC)\to \N\cup\{\infty\},$$
of the category $\Flow(\LC)$ of topological flows.  Therefore, the aim is to design a {\it corank function} 
\begin{equation*}
\mathrm{cork}\colon\Flow(\LC)\to\N\cup\{\infty\}
\end{equation*} 
 for topological flows able to satisfy the following equality: $$\mathrm{cork}(V,\phi)=\ent^*(V,\phi),\quad\forall\ (V,\phi)\in\mathrm{ob}(\Flow(\LC)).$$ 
 The initial motivation is the well-know characterisation of the algebraic entropy  of algebraic flows (i.e., pairs of type $(W,\psi)$ consisting of a discrete $\K$-space $W$ and an endomorphism $\psi$ of $ W$) in terms of the torsion-free rank of modules over PID; a detailed proof of this characterisation is given in \cite{GBSalce}. Since $\LC$ and $\Vect$ are known to be dual each other by  Lefschetz duality (see \S 2.1, 2.2), it is worth looking for the topological counterpart of the characterisation proved in \cite{GBSalce}.
 
 \medskip
 
After a few pages of preliminaries, we essentially start with Section~\ref{s:counting} that is dedicated to topological flows of non-zero topological entropy. There we give an accurate description of the role played by left Bernoulli shifts in order to get $\ent^*(V,\f)>0$. The main idea is to look at those closed $\phi$-invariant linear subspaces $S$ of $V$ such that $\phi$ induces a Bernoulli-like action on the factor $V/S$. Then $\ent^*(V,\f)$ is not zero only when $S$ turns out to be a $\phi$-cotrajectory $C(\f,U)=U\cap \f^{-1}U\cap\ldots\cap\f^{-n+1}U\cap\ldots$ of a fairly large linearly compact open subspace $U$ of $V$; see Corollary~\ref{cor:exist}. Such a cotrajectory has to be both {\it cocyclic} (i.e., $\dim_\K(V/U)=1$) and {\it non-stationary} (i.e., $C(\f,U)$ is not open in $V$). But one can say even more:
\begin{thmA}
For  a topological flow $(V,\f)$ the following are equivalent:
\begin{itemize}
\item[a)] the topological entropy $\ent^*(V,\f)=k\in\N$,
\item[b)] there exist $k\in\N$ coindependent non-stationary cocyclic $\f$-cotrajectories $\{C(\f,U_i)\mid i=1,\ldots,k\}$ such that $$\ent^*(C,\f\restriction_C)=0,\quad\text{where}\quad C=\bigcap_{i=1}^kC(\f,U_i).$$
\end{itemize}
Moreover, $\ent^*(V,\f)=\infty$ if, and only if, there exists an infinite (countable) coindependent family of non-stationary cocyclic $\f$-cotrajectories.
\end{thmA}
 In other words, a topological flow $(V,\phi)$ is ``essentially" a product of one-dimensional left Bernoulli shifts as many as $\ent^*(V,\f)$ counts. This insight of the relationship between Bernoulli shifts and non-zero entropy topological flows will lead us towards all the remaining results of the paper.
 
 \medskip
 
Here ``essentially" means that $(V,\f)$ is considered up to topological subflows of topological entropy $=0$. In Section~\ref{s:pinsker}, we detail how the vanishing of the topological (resp. algebraic) entropy  defines a torsion theory over the category of topological (resp. algebraic) flows. It is important to highlight that we cannot  provide here a module-theoretic property able to capture the vanishing of the topological entropy even if such module-theoretic property does exist in the algebraic case (namely, the torsion-free rank). For this reason the construction of the torsion theory over the category of topological flows cannot be made (at the moment) independent on the topological entropy itself. 

However, treating the notion of torsion axiomatically has the advantage to let us prove easily the existence of Pinsker-like constructions in our setup (see \cite{BL,DGB2} for this kind of constructions in other contexts). 
Namely, any topological flow $(V,\f)$ admits the largest factor 
$(P_{top}(V,\f),\overline\f^t)$ 
 of zero topological entropy, called {\it Pinsker factor}, and any algebraic flow $(W,\psi)$ admits the {\it Pinsker subflow} $(P_{alg}(W,\psi),\psi\restriction_a)$ that is the largest subflow of algebraic entropy $=0$. Since the torsion theories over $\LC$ and $\Vect$ turns out to be dual each other, one deduces the following relation between Pinsker factor and Pinsker subflow:
\begin{thmB} Let $(V,\f)$ be a topological flow. Then 
$$ (P_{top}(V,\phi),\overline\f^t)\cong (V/P_{alg}(\widehat V,\widehat\phi)^\perp,\overline\phi^a),$$
where $\overline\f^t\colon P_{top}(V,\phi)\to P_{top}(V,\phi)$ and $\overline\phi^a\colon V/P_{alg}(\widehat V,\widehat\phi)^\perp\to V/P_{alg}(\widehat V,\widehat\phi)^\perp$ are both induced by $\phi$. In particular,  $$D_+(V,\f)\cong_{top} P_{alg}(\widehat V,\widehat \f)^\perp\quad\text{and}\quad\widehat{D_+(V,\f)}\cong F_+(\widehat V,\widehat\f).$$
\end{thmB}
The $\K$-vector spaces $D_+(V,\f)$ and $F_+(\widehat V,\widehat\f)$ denote respectively the {\it domain of completely positive topological entropy of $(V,\f)$} and the {\it factor of completely positive algebraic entropy of $(\widehat V,\widehat\f)$} that will be defined in Section \ref{s:pinsker}.

Notice that the analogue \cite[Thm.~A]{DGB} of Theorem~B  in the context of Pontryagin duality  is dissimilarly deduced by the Bridge Theorem connecting the topological entropy and the algebraic entropy of abelian groups by means of Pontryagin duality. Notice that, since a version of the Bridge Theorem for locally linearly compact spaces is also available  in \cite{CGB2}, one can choose to use it in order to shorten some proofs in this paper.
However we prefer to avoid any possible application of the Bridge Theorem here in order to first make the proofs less mysterious and then provide an alternative proof of  the Bridge Theorem itself as an application of Theorem~C at the end of the last section; this important result finds some analogues in many other contexts (see, e.g., \cite{Weiss, Peters2, DGB, DGB-BT, Virili-BT}).

 We conclude the paper with the section devoted to the definition of the corank function over $\Flow(\LC)$ characterising $\ent^*$. Such a function has the purpose to measure the torsion of a topological flow $(V,\f)$, i.e., its domain of completely positive topological entropy $D_+(V,\f)$. 
 Subsequently, the corank of a topological flow $(V,\f)$ is defined to be a dimension of $(D_+(V,\f),\f\restriction_{D_+(V,\f)})$ regarded as a lattice: {\it the corank of $(V,\phi)$ is the dual Goldie dimension of the modular lattice formed by all closed $\f$-invariant subspaces of $D_+(V,\f)$}. 
 
 The idea to use the dual Goldie dimension to measure $D_+(V,\f)$ comes from  the observation that  the rank of a torsion-free module over a PID coincides with its Goldie dimension. Therefore, the authors of \cite{GBSalce} actually prove that the algebraic entropy of an algebraic flow coincides with the Goldie dimension of its factor of completely positive algebraic entropy. By the way, it is surprising how the dual Goldie dimension plays a central role in this context, since it usually does not  appear in the setting of topological modules. 

ery locally linearly compact space $V$ together with a continuous endomorphism $\phi$  corresponds univocally to a topological $\K[t]$-module ${}_\f V$, where $\K[t]$ denotes the (discrete) polynomial ring in the variable $t$. For an algebraic flow $(V,\f)$, the $\K[t]$-module ${}_\f V$ is actually discrete and we recover the well-known equivalence of categories $\Flow(\Vect)\cong{}_{\K[t]}\mathrm{Mod}.$
Finally, as a consequence of Theorem A, one proves
\begin{thmC}
For every topological flow $(V,\phi)$, one has
$$\ent^*(V,\phi)=\mathrm{cork}(V,\f).$$
\end{thmC}
\section{Preliminaries}\label{s:preliminaries}
\subsection{Linearly compact vector spaces}\label{ss:lc}
Let $\K$ be a discrete field. A Hausdorff topological $\K$-space $V$ is said to be \emph{linearly topologized} if it admits a neighborhood basis at $0$ consisting of $\K$-subspaces of $V$. A linearly topologized $\K$-space $V$ is \emph{linearly compact}  if any collection of closed linear varieties of $V$ with the finite intersection property has non-empty intersection. E.g., discrete vector spaces are linearly compact as well as compact linearly topologized spaces.

Here we collect a few properties that we frequently use further on with no previous acknowledgement. 
Let $V, V'$ and $W$ be linearly topoligized $\K$-spaces such that $W\leq V$, then
\begin{enumerate}[(LC1.)]
\item if $W$ is linearly compact, then $W$ is closed in $V$;
\item if $V$ is linearly compact and $W$ is closed, then $W$ is linearly compact;
\item linear compactness is preserved by continuous homomorphisms;
\item if $V$ is discrete, then $V$ is linearly compact if and only if $V$ has finite dimension;
\item if $W$ is closed, then $V$ is linearly compact if and only if $W$ and $V/W$ are linearly compact;
\item an inverse limit of linearly compact vector spaces is linearly compact;
\item every linearly compact vector space is complete;
\item for $V$ linearly compact, any continuous homomorphism  $f\colon V\to V'$ is open onto its image;
\item the isomorphism theorems hold in $\LC$;
\item every linearly compact $\K$-space is a Tychonoff product of one-dimensional $\K$-spaces.
\end{enumerate}
A topological $\K$-vector space $V$ is  \emph{locally linearly compact} (briefly, l.l.c.\,) if it admits a neighbourhood basis at $0$ consisting of linearly compact open $\K$-subspaces. Let $\LLC$ denote the category of l.l.c. $\K$-spaces and their continuous homomorphisms.
Clearly, $\LC$ and $\Vect$ (i.e., the category of discrete $\K$-vector spaces) are full subcategories of $\LLC$.
\subsection{Lefschetz Duality}\label{ss:lefdual}
For an l.l.c.\! space $V$, the Lefschetz dual $\widehat V$ is defined to be $\mathrm{CHom}(V,\K)$ endowed with the (linearly compact)-open topology. Namely, the topology on $\widehat V$ has the family
\begin{equation}
\{A^\perp\mid A\leq V,\ A\ \text{linearly compact}\}
\end{equation}
as neighbourhood basis at $0$, where $A^\perp=\{\chi\in\mathrm{CHom(V,\mathbb K)}:\chi(A)=0\}$. Given a closed $\K$- subspace $U$ of $V$, one has
\begin{equation}\label{eq:dual}
\widehat{V/U}\cong_{top} U^\perp\quad\text{and}\quad\widehat U\cong_{top}\widehat V/U^\perp.
\end{equation}
Lefschetz dual of discrete $\K$-spaces is linearly compact, and vice-versa.
More generally,  the dual functor $$\widehat{-}:\ \LLC\to\LLC$$  is such that $\doublehat{-}\colon\LLC\,\to\, \LLC$  is naturally ismorphic to $\mathrm{id}\colon\LLC\,\to\,\LLC$.
\subsection{Topological entropy and algebraic entropy over $\LLC$}\label{ss:entropies}
Let $V$ be an l.l.c. $\K$-space and $\B(V)$ denote the set of all linearly compact open $\K$-subspaces of $V$. Inspired by the notion of topological entropy for totally disconnected locally compact groups and their endomorphisms (see \cite{DSV,GBVirili}),  the {\it topological entropy} for l.l.c. $\K$-spaces is defined to be
\begin{equation}\label{eq:top}
\ent^*(V,\f)=\sup_{U\in\BV}H^*(\f,U),\quad (V,\f)\in\textup{ob}(\Flow(\LLC)), 
\end{equation}
where $H^*(\f,U)=\lim_{n\to\infty}\frac{1}{n}\cdot\dim(U/(U\cap\f^{-1}U\cap\ldots\cap\f^{-n+1}U))$. 

Dually,  the {\it algebraic entropy} of $(V,\f)\in\textup{ob}(\Flow(\LLC))$ is given by
\begin{equation}\label{eq:alg}
\ent(V,\f)=\sup_{U\in\BV}H(\f,U),
\end{equation}
where $H(\f,U)=\lim_{n\to\infty}\frac{1}{n}\cdot\dim((U+\f U+\cdots+\f^{n-1}U)/U)$. 

Both $\ent^*$ and $\ent$ are invariants of the category  $\Flow(\LLC)$ satisfying: 
\protect{
\begin{enumerate}[P1.]
\item {\it Invariance under conjugation}: for every topological isomorphism $\alpha\colon V\to V'$ one has 
$$\ent^*(V',\alpha\phi\alpha^{-1})=\ent^*(V,\phi)\quad\text{and}\quad\ent(V',\alpha\phi\alpha^{-1})=\ent(V,\phi);$$
\item {\it Monotonicity}: for any closed linear subspace  $W$ of $V$ such that $\f(W)\leq W$ one has
 $$\ent^*(V,\f)\geq \max\{\ent^*(W,\f\restriction_W),\ent^*(V/W,\overline\f)\}\ \text{and}$$
 $$\ent(V,\f)\geq \max\{\ent(W,\f\restriction_W),\ent(V/W,\overline\f)\},$$
 where $\overline\f\colon V/W\to V/W$ is induced by $\f$;
 \item {\it Logarithmic law}: if $k\in\N$, then 
 $$\ent^*(V,\phi^k)=k \cdot \ent^*(V,\phi)\quad\text{and}\quad\ent(V,\phi^k)=k \cdot \ent(V,\phi);$$
\item {\it Continuity on inverse limits}: let $\{W_i\mid i\in I\}$ be a directed system (under inverse inclusion) of closed $\f$-invariant linear subspaces of $V$. If $V=\varprojlim V/W_i$, then $\ent^*(\phi)=\sup_{i\in I} \ent^*(\overline\phi_{W_i})$, where any $\overline\f_{W_i}\colon V/W_i\to V/W_i$ is the continuous endomorphism induced by $\f$;
\item {\it Continuity on direct limits}: let $V$ be an l.l.c.\! vector space and $\phi\colon V\to V$ a continuous endomorphism. Assume that $V$ is the direct limit of a family $\{V_i\mid i\in I\}$ of closed $\phi$-invariant linear subspaces of $V$, and let $\phi_i=\phi\restriction_{V_i}$ for all $i\in I$. Then $\ent(\phi)=\sup_{i\in I}\ent(\phi_i);$
\item {\it Vanishing property}: $\ent^*\restriction_{\Flow(\Vect)}=0\quad\text{and}\quad\ent\restriction_{\Flow(\LC)}=0;$
 \item {\it Addition Theorem}: for every closed linear subspace $W$ of $V$  such that $\f(W)\leq W$ one has
$$\ent^*(V,\f)=\ent^*(W,\f\restriction_W)+\ent^*(V/W,\overline\f),$$
$$\ent(V,\f)=\ent(W,\f\restriction_W)+\ent(V/W,\overline\f),$$
where $\overline\f\colon V/W\to V/W$ is induced by $\f$.
\end{enumerate}

Moreover, since  it has been pointed out in \cite{CGB1} that $\ent\restriction_{\Flow(\Vect)}$ coincides with the dimension entropy $\ent_{dim}$ studied in \cite{GBSalce}, throughout this paper  all the well-known results concerning $\ent_{dim}$ will be automatically considered to be properties of the algebraic entropy $\ent\restriction_{\Flow(\Vect)}$.
\begin{rem}
The main ingredient in both definitions  is the fact that  $U\cap\phi(U)$ has finite codimension in $\phi(U)$ (or, equivalently, $U \cap \phi^{-1}(U)$ has finite codimension in $U$). This is a phenomenon - giving
the name {\it $\phi$-inert} to $U$ - inspired by a series of papers of U. Dardano and coauthors (see the survey
\cite{DDR}).
\end{rem}
\begin{ex}[Bernoulli shifts]\label{ex:bernoulli} {\normalfont For the purposes of this paper, we consider here only one type of Bernoulli shift, which will be referred to as {\it 1-dimensional Bernoulli shift}.

 Let $V_c$ be the linearly compact product $\prod_{n=0}^\infty \K$. The \emph{left 1-dimensional Bernoulli shift} is 
$${}_\K\beta\colon V_c\to V_c, \quad (x_n)_{n\in\N}\mapsto (x_{n+1})_{n\in\N}.$$
Therefore, $(V_c,{}_\K\beta)$ is a topological flow and $\ent^*(V_c,{}_\K\beta)=1$.

Dually, let $V_d=\bigoplus_{n=0}^\infty \K$ be equipped with the discrete topology. Thus one defines the \emph{right  (1-dimensional) Bernoulli shift} as
$$\beta_\K\colon V_d\to V_d, \quad (x_0,x_1,\ldots)\mapsto (0,x_0,x_1,\dots).$$
In particular, 
$$\widehat{(V_c,{}_\K\beta)}=(\widehat{V_c},\widehat{{}_\K\beta})\cong (V_d,\beta_\K),$$}
and  $(V_d,\beta_\K)$ has algebraic entropy $=1$.
\end{ex}
\section{Counting left one-dimensional Bernoulli shifts}\label{s:counting}
\subsection{Coindependent families of linearly compact spaces}\label{ss:coind} Let $M$ be a module over an associative ring with unit. A finite set $\{N_i\mid i\in I\}$ of proper submodules of $M$ is said to be {\it coindependent} if $$N_i+\bigcap_{j\neq i} N_j=M\quad\forall i\in I.$$ More generally, an arbitrary set of proper submodules of $M$ is coindependent if every finite subset is.
When $I=\N$, it suffices to check coindependence of $\{N_i\mid 1\leq i\leq n\}$ for each $n$.

Given  a coindependent finite set $\{N_0,\dots,N_k\}$ of proper submodules of $M$,  one has a canonical isomorphism of modules
\begin{equation}\label{eq:fin}
\xi_k\colon\frac{M}{N_0\cap\ldots\cap N_k}\longrightarrow\prod_{i=0}^k \frac{M}{N_i},
\end{equation}
where $m+N_0\cap\ldots\cap N_k\mapsto (m+N_0,\ldots,m+N_k),$ for all $m\in M.$ When $M$ is a linearly compact $\K$-space and each $N_i$ is closed in $M$, then \eqref{eq:fin} is a topological isomorphism because of the open mapping theorem (cf. LC8. in $\S\ref{ss:lc}$). 
\begin{prop}\label{prop:number}
Let  $(V,\f)$ be a topological flow and $\{W_0,\dots,W_n\}$ a coindependent set of closed $\f$-invariant linear subspaces of $V$. For $W=W_0\cap\ldots\cap W_n$ one has $$n+1\leq\ent^*(V/W,\overline\f)\leq\ent^*(V,\f)$$ whenever $\ent^*(\overline\f^i)>0$ for all $i=1,\ldots,n$, where each $\overline\f^i\colon V/W_i\to V/W_i$ is induced by $\f$.
\end{prop}
\begin{proof}
 By \eqref{eq:fin}, the following diagram
$$\xymatrix{\frac{V}{W}\ar[r]^{\overline\f}\ar[d]_{\xi_n}&\frac{V}{W}\ar[d]^{\xi_n}\\
\prod_{i=0}^n\frac{V}{W_i}\ar[r]^{\prod\overline\f^i}&\prod_{i=0}^n\frac{V}{W_i}}$$
commute, where $\overline\f\colon V/W\to V/W$ and $\overline\f^i\colon V/W_i\to V/W_i$ (for all $i$) are induced by $\f$. Since $\ent^*$ takes values in $\N\cup\{\infty\}$, the Addition Theorem together with the invariance under conjugation yields 
$$n+1\leq\ent^*(V/W,\overline\f)\leq\ent^*(V,\f).$$
\end{proof}
\subsection{Non-stationary cocyclic $\f$-cotrajectory}
Let $(V,\f)$ be a topological flow. For $U$ linearly compact open $\K$-subspace in $V$ and $n\in\N_{> 0}$,
set 
\begin{equation}
C_n(\f,U)=U\cap\f^{-1}U\cap\ldots\cap\f^{-n+1}U\in\BV.
\end{equation}
 Then one defines the {\it $\f$-cotrajectory of $U$} to be
\begin{equation}
C(\f,U)=\bigcap_{n\in\N_{>0}} C_n(\f,U).
\end{equation}
A $\f$-cotrajectory $C(\f,U)$ is said to be {\it cocyclic} if $U$ has codimension 1 in $V$, and $C(\f,U)$ is said to be {\it non-stationary} if $C(\f,U)$ is not open in $V$, i.e., the descending chain of linearly compact open spaces
$$U\geq\cdots\geq C_n(\f,U)\geq C_{n+1}(\f,U)\geq \cdots,$$
is not stationary (see \cite[Proposition~3.6]{CGB2}). Notice that a $\f$-cotrajectory is stationary whenever there exists some $n$ such that $C_n(\f,U)= C_{n+1}(\f,U)$ by \cite[Proposition~3.2]{CGB2}.
\begin{lem}\label{lem:cod1} Let $(V,\f)$ be a topological flow and $C(\f,U)$ a non-stationary cocyclic $\f$-cotrajectory. Then
\begin{enumerate}
\item the closed subspace $\f^{-n}U$ has codimension 1 in $V$ for every $n\in N$;
\item $\mathfrak F=\{\f^{-n}U\}_{n\in\N}$ is a coindependent family in $V$;
\item the topological entropy of $\f$  with respect to $U$ is equal to 1, i.e., $H^*(\f,U)=1$.
\end{enumerate}
\end{lem}
\begin{proof}
Let $\pi\colon V\to V/U$ be the canonical projection. For all $n\in\N$, the rank of the linear map
$\pi\circ\f^n\colon V\to V/U$ is less or equal to $\dim(V/U)$. Hence $\dim(V/\f^{-n}U)\leq 1$ since $C(\f,U)$ is cocyclic. Moreover, for $C(\f,U)$ is non-stationary,  $C_n(\f,U)\neq C_{n+1}(\f,U)$ for all $n\in\N$. Thus $\f^{-n}U\neq V$ for all $n\in\N$, and $(1)$ holds.

In order to prove  $(2)$, by \cite[Lemma~1.1]{dgd}, it suffices to show that $\f^{-n}U+C_n(\f,U)=V$ for each $n\in\N_{>0}$. This follows easily by $(1)$ and the fact that $C(\f,U)$ is non-stationary.

Let $n\in\N_{>0}$. Thus $(1)$ and $(2)$ implies
$$\dim\Big(\frac{C_n(\f,U)}{C_{n+1}(\f,U)}\Big)=\dim\Big(\frac{C_n(\f,U)+\f^{-n}U}{\f^{-n}U}\Big)=\dim\Big(\frac{V}{\f^{-n}U}\Big)=1,$$
and \cite[Proposition~3.2]{CGB2} yields (3).
\end{proof}
\begin{cor}\label{cor:exist}
Let $(V,\f)$ be a topological flow. Thus $\ent^*(\f)>0$ if, and only if,  there  exists a non-stationary cocyclic $\f$-cotrajectory.
\end{cor}
\begin{proof} Suppose $\ent^*(\f)>0$. Thus there exists $U\in\BV$ such that $C(\f,U)$ is not open (see \cite[\protect{Proposition~3.6}]{CGB2}). For $U\in\BV$, $\dim(V/U)=n<\infty$. Let $V=U\oplus\langle v_1,\ldots,v_n\rangle$ and $U_i=U\oplus \langle v_1,\ldots,\hat v_i,\ldots,v_n\rangle\in\BV$. Thus $U=U_1\cap\ldots\cap U_n$ and $C(\f,U)=\bigcap_{i=1}^n C(\f,U_i)$. Since $C(\f,U)$ is not open, there exists $i\in\{1,\ldots,n\}$ such that $C(\f,U_i)$ is not open, i.e., $C(\f,U_i)$ is a non-stationary cocyclic $\f$-cotrajectory.

Conversely, if there exists a non-stationary cocyclic cotrajectory $C(\f,U)$, then $1=H^*(\f,U)\leq\ent^*(V,\f)$ by Lemma~ \ref{lem:cod1}(3).
\end{proof}
\begin{prop}\label{prop:cotraj coind} Let $(V,\f)$ be a topological flow.
If $C(\f,U)$ is a non-stationary cocyclic $\f$-cotrajectory of $V$, then one has a topological isomorphism
$$\theta\colon\frac{V}{C(\f,U)}\longrightarrow \prod_{n\geq 0}\frac{V}{\f^{-n}U}.$$
Moreover the induced map $\overline\f\colon V/C(\f,U)\to V/C(\f,U)$ is conjugated to a left Bernoulli shift and $\ent^*(\overline\f)=1$.
\end{prop}
\begin{proof} By the Lemma above, $\mathfrak F=\{\f^{-n}U\}_{n\in\N}$ is a coindependent set of closed proper linear subspaces of $V$. 
By \eqref{eq:fin}, there are continuous surjective maps
\begin{equation}
\xi_n\circ\pi_n\colon V\to\prod_{i=0}^n \frac{V}{\f^{-n}U},
\end{equation}
where $\pi_n\colon V\to V/C_{n+1}(\f,U)$ is the canonical projection.
By the universal property of inverse limit, one obtains the continuous surjective map
\begin{equation}\label{eq:inf}
\varprojlim (\xi_n\circ\pi_n)\colon V\longrightarrow\prod_{n\in\N} \frac{V}{\f^{-n}U},
\end{equation}
(see \cite[Lemma 1.1.7]{profinite}). By the open mapping theorem for linearly compact spaces, the continuous surjection \eqref{eq:inf} gives rise to the following topological isomorphism
\begin{equation}\label{eq:iso}
\theta\colon\frac{V}{C(\f,U)}\longrightarrow \prod_{n\geq 0}\frac{V}{\f^{-n}U},\quad v+C(\f,U)\mapsto (v+\f^{-n}U)_{n\in\N}.
\end{equation}
Let $e_0$ be a non-zero vector of $V/U$. Since $C(\f,U)$ is non-stationary and cocyclic, $\f$ induces an isomorphism $\varphi_n$ from $V/\f^{-n}U$ to $V/\f^{-n+1}U$ for all $n\in\N$. Thus, for $n\in \N$, let $e_{n}\in V/\f^{-n}U$ such that $\varphi_n(e_{n})=e_{n-1}$. Let $\beta$ be the left Bernoulli shift of $\prod_{n\in\N} \K e_n$. As the diagram
\begin{equation}
\xymatrix{
\frac{V}{C(\f,U)}\ar[r]^-\theta\ar[d]_{\overline\f}&\prod_{n\in\N} \frac{V}{\f^{-n}U}\ar[r]^\cong&\prod_{n\in\N} \K e_n\ar[d]^\beta\\
\frac{V}{C(\f,U)}\ar[r]^-\theta&\prod_{n\in\N} \frac{V}{\f^{-n}U}\ar[r]^\cong&\prod_{n\in\N} \K e_n}
\end{equation}
commutes, $\overline\f$ is topologically conjugated to a left Bernoulli shift.
By \cite[Prop. 3.11(a), Ex. 3.13(a)]{CGB2}, $\ent^*(V/C(\f,U),\overline\f)=1$, and this concludes the proof.
\end{proof}
\begin{rem} For an arbitrary module $M$ is not true that if $M$ has an infinite coindependent set of proper submodules, then $M$ has an homomorphic image given by a direct product of infinitely many non-zero modules. E.g., let $\Z$ be the abelian group of integers. Then $\{p\Z\mid p\ \text{prime}\}$ is an infinite coindependent set of proper $\Z$-submodules, but there is no homomorphic image of $\Z$ onto a direct product of infinitely many non-zero abelian groups. Therefore, the linear compactness is an essential ingredient in the proof of Proposition \ref{prop:cotraj coind}.
\end{rem}
\begin{cor}\label{cor:zero ent} Let $(V,\f)$ be a topological flow. Thus $\ent^*(V,\f)=0$ if, and only if, $\ent(\widehat V,\widehat\f)=0$.\end{cor}
\begin{proof} Suppose $\ent^*(V,\f)>0$. By the latter proposition, there exists a factor of $V$ acted on by $\overline\f$ as left Bernoulli shift. By Lefschetz duality, $\widehat V$ admits a $\widehat\f$-invariant linear subspace acted on by the restriction of $\widehat\f$ as right Bernoulli shift (see Example \ref{ex:bernoulli}). By the monotonicity of $\ent$, one concludes that $1\leq\ent(V,\f)$. 

Vice-versa, suppose $\ent(\widehat V,\widehat\f)>0$. By \cite[Theorem 4.3]{GBSalce},  $\widehat V$ admits a $\widehat\f$-invariant linear subspace acted on by $\widehat\f$ as right Bernoulli shift (namely an infinite-dimensional $\f$-trajectory). A similar reasoning as above concludes the proof.\end{proof}
\begin{cor}
Let $(V,\f)$ be a topological flow. Thus $\ent^*(\f)=1$ if, and only if, there exists a non-stationary cocyclic $\f$-cotrajectory $C(\f,U)$ such that $\ent^*(\f\restriction_{C(\f,U)})=0$.
\end{cor}
\subsection{Proof of Theorem A}\label{ss:a} Now we are in position to prove Theorem~A that can be considered to be  dual of \cite[Theorem 4.7]{GBSalce}, where the algebraic entropy of an algebraic flow is expressed in terms of right Bernoulli shifts. We just need one more technical lemma.
\begin{lem}\label{lem:entres} Let $(V,\f)$ be a topological flow and $W$ a closed $\f$-invariant linear subspace of $V$. If $\ent^*(W,\f\restriction_W)>0$,  then there exists a non-stationary cocyclic $\f$-cotrajectory $C(\f,U)$ such that $C(\f,U)+W=V$.
\end{lem}
\begin{proof} Firstly, notice that  $\B(W)=\{U\cap W\mid U\in\BV\}$. 
Since $0<\ent^*(W,\f\restriction_W)$, by Corollary~\ref{cor:exist}, there exists $U\in\BV$ such that   $C(\f\restriction_W, U\cap W)=C(\f,U)\cap W$ is non-stationary  and cocyclic. Therefore, $C(\f,U)$ is also cocyclic and non-stationary. In particular, both
the topological flows $(V/C(\f,U),\overline\f)$ and $(W/C(\f\restriction_W,U\cap W),\overline{\f\restriction_W})$ can be regarded as left 1-dimensional Bernoulli shift. Now, the topological flow $(Z=(C(\f,U)+W)/C(\f,U),(\overline\f)\restriction_Z)$ is isomorphic to $(W/C(\f\restriction_W,U\cap W),),\overline{\f\restriction_W})$ by LC8; namely, the following diagram
\begin{equation}
\xymatrix{Z=\frac{C(\f,U)+W}{C(\f,U)}\ar[d]_{(\overline\f)\restriction_Z}\ar[r]^\cong & \frac{W}{C(\f\restriction_W,U\cap W)}\ar[d]^{\overline{\f\restriction_W}}\\
Z=\frac{C(\f,U)+W}{C(\f,U)}\ar[r]_\cong& \frac{W}{C(\f\restriction_W,U\cap W)}}
\end{equation}
commutes. Thus
 $\ent^*(Z,(\overline\f)\restriction_Z)=1$, since $\ent^*(W/C(\f\restriction_W,U\cap W),),\overline{\f\restriction_W})=1$. Finally, by Example~\ref{ex:bern shift}, 
$V/C(\f,U)= Z,$ i.e., $C(\f,U)+W=V$.
\end{proof}
\begin{rem} The reference to Example~\ref{ex:bern shift} could create the appearance of a 
circular argument  but requiring the reader read forward only once simplifies the organization of the paper.
\end{rem}
\begin{thmA}
Let $(V,\f)$ be a topological flow. Then the following are equivalent:
\begin{itemize}
\item[a)] the topological entropy $\ent^*(\f)=k\in\N$,
\item[b)] there exist $k\in\N$ coindependent non-stationary cocyclic $\f$-cotrajectories $\{C(\f,U_i)\mid i=1,\ldots,k\}$ such that $\ent^*(\f\restriction_C)=0$ where $C=C(\f,U_1)\cap\ldots\cap C(\f,U_k)$
\end{itemize}
Moreover, $\ent^*(V,\f)=\infty$ if, and only if, there exists an infinite (countable) coindependent family of non-stationary cocyclic $\f$-cotrajectories.
\end{thmA}
\begin{proof}[Proof of Theorem~A] By Corollary~\ref{cor:exist}, we can assume $k>0$.  By the Addition Theorem, Proposition~\ref{prop:number} and Proposition~\ref{prop:cotraj coind}, b) implies a).

  Now, since $k>0$, there is at least one non-stationary cocyclic $\f$-cotrajectory. By Proposition~\ref{prop:number} and Proposition~\ref{prop:cotraj coind}, every finite set of coindependent non-stationary cocyclic $\f$-cotrajectories has cardinality less or equal to $k$.   
  Let $m$ be the maximal number of coindependent non-stationary cocyclic $\f$-cotrajectories, thus $1\leq m\leq k$.

Now suppose $m<k$ and let $\{C_i\mid 1\leq i\leq m\}$ be a maximal family of coindependent non-stationary cocyclic $\f$-cotrajectories. By the Addition Theorem, $\ent^*(\f\restriction_C)>0$, where $C=C_1\cap\ldots\cap C_m$. Applying Lemma~\ref{lem:entres} yields a non-stationary cocyclic $\f$-trajectory $C_{m+1}$ such that $C_{m+1}+C=V$, and this contradicts the maximality of $m$, i.e., a) implies b).

Finally, let $\ent^*(V,\f)=\infty$. Let $C_1$ be a non-stationary cocyclic $\f$-cotrajectory (see Corollary \ref{cor:exist}). Since $\ent^*(V/C_1,\overline\f)=1$, the Addition Theorem yields $\ent^*(C_1,\f\restriction_{C_1})=\infty$. By Lemma \ref{lem:entres}, there exists a non-stationary cocyclic $\f$-cotrajectory $C_2$ such that $\{C_1,C_2\}$ is coindependent. Repeating the procedure at each step yields the claim. Conversely, since every finite subfamily of a countable coindependent family is coindependent as well, Proposition~\ref{prop:number} concludes the proof.
\end{proof}
\section{Torsion theories and Pinsker-like constructions}\label{s:pinsker}
\subsection{Torsion theory for abelian categories}\label{ss:torsion}
Following \cite{torsion}, a {torsion theory} for a subcomplete abelian category $\C$ consists of a couple $\tau=(\ca T,\ca F)$ of classes of objects of $\C$ satisfying 
\begin{itemize}
\item[(I)] $\ca T\cap\ca F =\{0\}$.
\item[(II)] If $T\to A\to 0$ is exact with $T\in\ca T$, then $A\in \ca T$.
\item[(III)] If $0\to A\to F$ is exact with $F\in\ca F$, then $A\in\ca F$.
\item[(IV)] For each object $X$ of $\C$ there is an exact sequence
$0\to T\to X\to F\to 0,$
with $T\in\ca T$ and $F\in\ca F$.
\end{itemize}
Thus $\ca T$ is a {\it torsion class} and $\ca F$ is a {\it torsion-free class} of $\C$.  If both classes $\ca T$ and $\ca F$ satisfy (IV) and contain with any object all its isomorphic images, then axioms (I) through (III) can be replaced by the {\it ortogonality axiom}
\begin{itemize}
\item[(V)] $\Hom(T,F)=0,\ \forall T\in\ca T, F\in\ca F$.
\end{itemize}
In particular, if $\ca T$ and $\ca F$ are complete with respect to the orthogonality axiom - i.e., if an object $T$ of $\C$ is such that $\Hom(T, F)=0$ for all $F\in\ca F$, then $T\in\ca T$; and the other way round - then $\tau=(\ca T,\ca F)$ is a torsion theory (cf. \cite[Thm. 2.1]{torsion}). For further applications we state here the following result.
\begin{prop}[\protect{\cite[Proposition 2.4]{torsion}}]\label{prop:torsion}
Let $\tau=(\ca T,\ca F)$ be a torsion theory of $\C$, and let $M\in ob(\C)$. Then there is a unique largest subobject $M_t$ of $M$ which is a member of $\ca T$. Moreover, $M/M_t\in\ca F$. The object $M_t$ can be calculated by either of the equalities
\begin{enumerate}
\item $M_t=\bigcup\{T\subseteq M\mid T\in\ca T\}$, or
\item $M_t=\bigcap\{S\subseteq M\mid M/S\in\ca F\}.$
\end{enumerate}
\end{prop}
Hence any torsion theory over $\C$ is uniquely determined by its torsion class: $$\ca F=\{M\in ob(\C)\mid M_t=0\}.$$
\subsection{Pinsker factor of a topological flow} Since the category of topological flows  can be regarded as a functor category over $\LC$, $\Flow(\LC)$ is a complete abelian category since $\LC$ is.
 Let
\begin{equation}
\ca F_{top}=\{(F,\varphi)\in ob(\Flow(\LC))\mid\ent^*(F,\varphi)=0\}.
\end{equation}
For all the properties stated in $\S$\ref{ss:entropies}, $\ca F_{top}$ is closed under taking kernels, infinite products and extensions. Thus \cite[Thm. 3.2]{torsion} applies and $\ca F_{top}$ is a well-defined torsion-free class on $\Flow(\LC)$. The corresponding torsion class $\ca T_{top}$ can be obtained by the orthogonality axiom. Namely,
$ (T,\tau)\in \ca T_{top}$ if, and only if, 
\begin{equation}\label{eq:tor}
\Hom_{\Flow(\LC)}((T,\tau),(F,\varphi))=0,\ \forall (F,\varphi)\in \ca F_{top}.
\end{equation}
A topological flow $(V,\f)$ is said to have {\it completely positive topological entropy}, $\ent^*(V,\f)>>0$, if $\ent^*(V/Z,\overline\f)\neq0$ for every closed $\f$-invariant proper linear subspace $Z$ of $V$. Equivalently, $\ent^*(V,\f)>>0$ whenever $(V,\f)$ does not admit non-trivial factors of zero topological entropy.
\begin{prop}
$\ca T_{top}=\{ (V,\f)\in ob(\Flow(\LC))\mid\ent^*(V,\f)>>0\}\cup\{0\}.$
\end{prop}
\begin{proof}
Let $(T,\tau)\in T_{top}$ and suppose it admits a non-trivial factor $(T',\tau')$ of zero topological entropy. Thus $0\neq (T',\tau')\in\ca F_{top}$ and $\Hom_{\Flow(\LC)}((T,\tau),(T',\tau'))\neq 0$ contradicting \eqref{eq:tor}.

Conversely, let $\ent^*(V,\f)>>0$. Suppose there is $(F,\varphi)\in \ca F_{top}$ such that $$\Hom_{\Flow(\LC)}((V,\f),(F,\varphi))\neq 0.$$ Thus $(V,\f)$ admits a non-trivial factor $(V',\f')$ isomorphic to a subflow of $(F,\varphi)$ by LC9. In particular, $0\leq\ent^*(V',\f')\leq\ent^*(F,\varphi)=0$; contradiction.
\end{proof}
\begin{prop}\label{prop:exist dom}
Let $(V,\f)$ be a  topological flow. Thus $V$ admits the largest closed $\f$-invariant $\K$-subspace $D_+(V,\f)$ such that $\ent^*(D_+(V,\f),\f\restriction_+)>>0$, where $\f\restriction_+:=\f\restriction_{D_+(V,\f)}$. Moreover, $$\ent^*(D_+(V,\f),\f\restriction_+)=\ent^*(V,\f).$$ The subspace $D_+(V,\f)$ can be computed by either of the equalities
\begin{enumerate}
\item $D_+(V,\f)=\bigcup\{Z\leq V\ \text{closed}\mid \f Z\leq Z,\ \ent^*((Z,\f\restriction_Z)>>0\}$, or
\item $D_+(V,\f)=\bigcap\{Z\leq V\ \text{closed}\mid \f Z\leq Z,\ \ent^*(V/Z,\overline\f)=0\}.$
\end{enumerate}
\end{prop}
\begin{proof} It follows by Proposition~\ref{prop:torsion} and the Addition Theorem.\end{proof}
We shall call $D_+(V,\f)$ the {\it domain of completely positive topological entropy} of $(V,\f)$. Clearly, $D_+(V,\f)=V$ if, and only if, $(V,\f)$ has completely positive topological entropy; and $D_+(V,\f)=0$ if, and only if, $\ent^*(V,\f)$=0.
Moreover the Addition Theorem together with Proposition \ref{prop:exist dom}(2) yields the following.
\begin{cor}\label{cor:ent dom} Given a topological flow $(V,\f)$, the domain $D_+(V,\f)$ is the smallest closed $\f$-invariant linear subspace of $V$ such that $$\ent^*(D_+(V,\f),\f\restriction_+)=\ent^*(V,\f).$$
Thus $(V,\f)$ admits the {\it Pinsker factor}, i.e., the largest factor $(P_{top}(V,\f),\overline\f^t)$ of $(V,\f)$ having zero  topological entropy. Namely,  $P_{top}(V,\f)=V/D_+(V,\f)$ and $\overline\f^t$ is induced by $\f$.
\end{cor}
\begin{fact}\label{fact:inv}
Let $(V,\f)$ be a topological flow. If $V=D_+(V,\f)$, then  
\begin{itemize}
\item[(a)] $V$ has no open $\f$-invariant proper linear subspaces; 
\item[(b)] a finite  coindependent set of closed $\f$-invariant $\K$-subspaces of $V$ has cardinality $\leq\ent^*(V,\f)$.
 \end{itemize}
\end{fact}
\begin{proof} (a) follows by \cite[Proposition $3.11(b)$]{CGB2} and (b) is a consequence of Proposition~\ref{prop:number}.\end{proof}
\subsection{Pinsker subspace of an algebraic flow}\label{ss:pinsker subsp} By analogy with the topological case, $\Flow(\Vect)$ can be equipped with the following torsion theory:
\begin{align}\label{eq:algtor}
\ca T_{alg}=\{(W,\psi)\mid\ent(W,\psi)=0\},\quad\text{and}\\ 
\ca F_{alg}=\{(W,\psi)\mid\ent(W,\psi)>>0\}\cup\{0\},\nonumber
\end{align}
where $\ent(W,\psi)>>0$ denote an algebraic flow $(W,\psi)$ of {\it completely positive algebraic entropy}, i.e., $\ent(Z,\psi\restriction_Z)\neq 0$ whenever $Z$ is a non-trivial $\psi$-invariant $\K$-subspace of $W$. 
\begin{rem}\label{rem:alg tor} Recall that $\Flow(\Vect)$ is equivalent to ${}_{\K[t]}\mathrm{Mod}$, which denotes the category of  (abstract) modules over the polynomial ring in one variable. In \cite{GBSalce} it has been pointed out that the torsion theory defined by \eqref{eq:algtor} goes over into the usual torsion theory of $\K[t]$-modules (see  \cite[Theorem~4.3]{GBSalce}). 
\end{rem} 
\begin{prop}
Let $(W,\psi)$ be an algebraic flow. Then $W$ admits the largest $\psi$-invariant $\K$-subspace $P_{alg}(W,\psi)$ of zero algebraic entropy. Moreover, $P_{alg}(W,\psi)$ is given by either of the equalities
\begin{enumerate}
\item $P_{alg}(W,\psi)=\bigcup\{Z\leq W\ \text{closed}\mid \psi Z\leq Z,\ \ent((Z,\psi\restriction_Z)=0\}$, or
\item $P_{alg}(W,\psi)=\bigcap\{Z\leq W\ \text{closed}\mid \psi Z\leq Z,\ \ent(W/Z,\overline\psi)>>0\}.$
\end{enumerate}
Consequently every algebraic flow $(W,\psi)$ admits the {\normalfont factor of completely positive algebraic entropy} $F_+(W,\psi)=W/P_{alg}(W,\psi)$, i.e., the smallest factor of $W$ such that $$\ent(F_+(W,\psi),\overline\psi^+)=\ent(W,\psi),$$
where $\overline\psi^+\colon F_+(W,\psi)\to F_+(W,\psi)$ is induced by $\psi$.
\end{prop}
The linear subspace $P_{alg}(W,\psi)$ is called  {\it Pinsker subspace} of $(W,\psi)$ and the corresponding algebraic flow $(P_{alg}(W,\psi),\psi\restriction_a)$ is named {\it Pinsker subflow}. Clearly, an algebraic flow $(W,\psi)$ has completely positive algebraic entropy if, and only if, $P_{alg}(W,\psi)=0$.
\begin{rem} 
 An alternative proof of the existence of the Pinsker subspace can be provided by adapting the proof of \cite[Proposition 3.1]{DGB2}.
\end{rem}
\subsection{Duality of Pinsker-constructions}
Here we relate Pinsker factors and Pinsker subflows by Lefschetz duality. Firstly observe that $(\ca T_{top},\ca F_{top})$ goes over into the torsion theory $(\ca F_{top},\ca T_{top})$ of the opposite category $\Flow(\LC)^{op}$. Since the opposite category $\Flow(\LC)^{op}$ is equivalent to $\Flow(\Vect)$ by Lefschetz duality, one obtains a new torsion theory on $\Flow(\Vect)$ determined by
\begin{equation}
\ca T^*_{alg}=\{(V,\f)\in ob(\Flow(\Vect))\mid \ent^*(\widehat V,\widehat\f)=0\}.
\end{equation}
Since $\ent^*(\widehat V,\widehat\f)=0$ if, and only if, $\ent(V,\f)=0$ by Corollary~\ref{cor:zero ent}, one has $\ca T^*_{alg}=\ca T_{alg}$. In other words, the torsion theories $(\ca T_{top},\ca F_{top})$ and $(\ca T_{alg},\ca F_{alg})$ are dual one each other. Therefore one deduces by \eqref{eq:dual} the following analogue of \cite[Theorem~A]{DGB}.
\begin{thmA} Let $(V,\f)$ be a topological flow. Then 
$$ (P_{top}(V,\phi),\overline\f^t)\cong_{top} (V/P_{alg}(\widehat V,\widehat\phi)^\perp,\overline\phi^a),$$
where $\overline\f^t\colon P_{top}(V,\phi)\to P_{top}(V,\phi)$ and $\overline\phi^a\colon V/P_{alg}(\widehat V,\widehat\phi)^\perp\to V/P_{alg}(\widehat V,\widehat\phi)^\perp$ are both induced by $\phi$. In particular,  $$D_+(V,\f)\cong_{top} P_{alg}(\widehat V,\widehat \f)^\perp\quad\text{and}\quad\widehat{D_+(V,\f)}\cong F_+(\widehat V,\widehat\f).$$
Moreover, $\ent^*(V,\f)>>0$ if, and only if, $\ent(\widehat V,\widehat\f)>>0$.
\end{thmA}
\begin{ex}\label{ex:bern shift} {\normalfont In the notations of Example~\ref{ex:bernoulli}, the right 1-dimensional Bernoulli shift has completely positive algebraic entropy: every non-trivial $\beta_\K$-invariant $\K$-subspace of $V_d$ is isomorphic to $V_d$ since such a subspace has countable but not finite dimension being $\beta_\K$-invariant. 

Consequently,  the left 1-dimensional Bernoulli shift has completely positive topological entropy. In particular, a left 1-dimensional Bernoulli shift cannot admit any proper subflow of positive topological entropy by Proposition~\ref{prop:exist dom}(2).}
\end{ex}
\section{The corank of a topological flow}\label{s:corank}
\subsection{(Dual) Goldie dimension}\label{ss:dgd} Here we recall both the definitions of Goldie dimension and of dual Goldie dimension placing greater emphasis on dual Goldie dimension. These notions were firstly introduced in module theory, but  modular lattices seem to be the proper setting for their definition.

Following \cite{Facchini}, let $(L, \vee ,\wedge)$ denote a modular lattice - i.e., $x\vee(a\wedge b)=(x\vee a)\wedge  b$ whenever $x\leq b$ - with a smallest element 0 and a greatest element 1.  We start from a bit of terminology. A finite subset $\{a_i\mid i\in I\}$ of $L\setminus\{0\}$ is said to be {\it coindependent} if 
\begin{equation}
a_i\vee(\bigwedge_{j\neq i} a_j)=1\quad\forall i\in I.
\end{equation}
Every element $x$ of $L\setminus\{0\}$ (considered as a singleton $\{x\}$) is coindependent. An arbitrary subset $A$ of $L\setminus\{0\}$ is coindependent if all its finite subsets are coindependent. 
 In a lattice $L$ with $1$ an element $a\in L$ is called {\it superfluos} if $a\vee b\neq 1$ for all $b\neq 1$. A lattice $L\neq\{0\}$ is said to be {\it couniform} if every element $a\neq 1$ is superfluous. An element $a\neq 1$ of $L$ is called {\it couniform} if the sublattice $[a,1]=\{x\in L\mid a\leq x\}$ is couniform. Consequently one has the following well-known result.
 \begin{prop}[see \cite{Facchini}]\label{prop:goldie}The following conditions are equivalent for a modular lattice $L$ with $0$ and $1$.
 \begin{enumerate}
 \item $L$ does not contain infinite coindependent subsets.
 \item $L$ contains a finite coindependent set $\{a_1,\dots,a_n\}$ with $a_i$ couniform for every $i$ and $a_1\wedge a_2\wedge\ldots\wedge a_n$ superfluous in $L$.
 \item The cardinality of a coindependent subset of $L$ is $\leq m$ for an $m\in\N$.
 \end{enumerate}
 Moreover, if these  conditions hold and $\{a_1,\dots,a_n\}$ is a finite coindependent subset of $L$ with $a_i$ couniform for every $i$ and $a_1\wedge a_2\wedge\ldots\wedge a_n$ superfluous in $L$, then any other coindependent subset of $L$ has cardinality $\leq n$.
 \end{prop}
 Thus if the modular lattice $L$ contains infinite coindependent subsets, then $L$ is said to have {\it infinite dual Goldie dimension}; otherwise the (finite) {\it dual Goldie dimension} of $L$ is defined to be the maximal cardinality $\codi L$ of a finite coindependent subset of $L\setminus\{0\}$. Since the dual lattice $( L,\wedge,\vee)$ of a modular lattice $( L,\vee,\wedge)$ is also a modular lattice, one may define the {\it Goldie dimension} of $( L,\vee,\wedge)$ to be the dual Goldie dimension of the dual lattice  $( L,\wedge,\vee)$.
   \begin{fact}[\protect{\cite[Propositions 2.37 and 2.42]{Facchini}}]\label{fact:couniform} For a modular lattice $L$,  
   \begin{enumerate}
   \item $\codi L=0$  if, and only if, $L$ has exactly one element;
   \item  $\codi L=1$ if, and only if, $L$ is couniform.
   \end{enumerate}
   Moreover, the dual version of $(1)$ and $(2)$ holds for the Goldie dimension.
 \end{fact}
 Let $M$ be a module over an associative ring with unit. Now we apply the (dual) Goldie dimension of modular lattices just introduced to the lattice $(L(M),+,\cap)$ of all submodules of  $M$. Thus the {\it dual Goldie dimension} $\codi{M}$ of $M$ is defined to be the dual Goldie dimension of $L(M)$. In particular, $\codi{M}$ turns out to be the maximal cardinality of a finite coindependent family (see \S\ref{ss:coind}) of proper submodules of $M$ ( or $\infty$).  Analogously, the {\it Goldie dimension} $\dim(M)$ of $M$ is the Goldie dimension of  $({L}(M),+,\cap)$ and it coincides with the maximal $n$ (or $\infty$) such that there exists a direct sum $\bigoplus_{i=1}^n M_i\subseteq M$ of non-zero submodules of $M$.
  \begin{prop}[\protect{\cite[\S 2.7]{Facchini}}] \label{prop:hull} Let $E(M)$ be the injective hull of $M$. Then the following hold:
  \begin{itemize}
  \item[(a)] $\dim(M)=\dim(E(M))$;
  \item[(b)] $\dim(M)=n<\infty$ if, and only if, $E(M)$ is the direct sum of $n$ indecomposable modules.
  \end{itemize}
  \end{prop}
  \begin{rem}
  Since the category $\Flow(\Vect)$ is equivalent to ${}_{\K[t]}\mod$, Proposition~\ref{prop:hull} can be applied directly to algebraic flows.
  \end{rem}
  The notion of projective cover is dual to that of injective hull. Although each abstract module has an injective hull, projective covers seldom exist, therefore the latter result cannot be dualised to the dual Goldie dimension of abstract modules. Nevertheless such a statement holds over the dual category when it does exists.
 \subsection{Modular lattice of a topological flow}
 Let $(V,\f)$ be a flow over $\LC$. We can associate to $(V,\f)$ the lattice of all closed $\f$-invariant $\K$-subspaces of $V$, which will be denoted by 
 $(\Lambda(V,\f), +,\cap)$.
 By properties LC1.--LC3., the lattice $(\Lambda(V,\f), +,\cap)$ is modular. Therefore we are allowed to compute the dual Goldie dimension of  $(\Lambda(V,\f), +,\cap)$, i.e., the maximal cardinality of a set of coindependent  closed $\phi$-invariant linear subspaces of $V$. We denote it by $\codi(V,\f)$. 
 \begin{rem}
 Notice that, by Lefschetz duality, the dual lattice of $\Lambda(V,\f)$ is equivalent to the lattice consisting of all $\widehat\f$-invariant $\K$-subspaces of $\widehat V$, say $\Lambda(\widehat V,\widehat\f)$.
 \end{rem}
\begin{prop} For every discrete field $\K$ the following hold:
\begin{enumerate}
\item $\Flow(\LC)$ is an abelian category with enough injectives and projectives;
\item $\Flow(\LC)$ admits projective covers.
\end{enumerate} 
Moreover, for every topological flow $(V,\f)$ with projective cover $(P,\rho)$ one has
\begin{enumerate}
\setcounter{enumi}{2}
 \item $\codi(V,\f)=\codi(P,\rho)$;
\item $\codi(V,\f)=n<\infty$ iff $(P,\rho)$ is the product of exactly $n$ indecomposable topological flows.
\end{enumerate}
\end{prop}
\begin{proof}
It follows by the fact that $\Flow(\LC)$ is the Lefschetz dual of $\Flow(\Vect)$ - which is nothing but ${}_{\K[t]}\mod$ - and Proposition~\ref{prop:hull}.
\end{proof}
\subsection{The corank function}\label{ss:corank}
For a topological flow $(V,\phi)$, one may consider the domain of completely positive topological entropy $(D_+(V,\f),\f\restriction_+)$ - i.e., the torsion subflow of $(V,\f)$ -  and give the following definition.
\begin{defn}
The {\it corank} $\mathrm{cork}(V,\f)$ of a topological flow $(V,\phi)$ is defined to be the dual Goldie dimension of $(D_+(V,\f), \f\restriction_+). $ Thus one obtains a {\it corank function} 
$$\mathrm{cork}\colon\Flow(\LC)\to\N\cup\{\infty\}.$$
\end{defn}
Indeed, the corank is meant to measure the $\K$-subspace where the map $\f$ is actually creating disorder. By Fact~\ref{fact:couniform}, one easily deduces the following properties:
\begin{itemize}
\item[(C1)] $\mathrm{cork}(V,\f)=0$ if, and only if, $D_+(V,\f)=0$.
\item[(C2)] $\mathrm{cork}(V,\f)=1$ if, and only if, $\Lambda(D_+(V,\f),\f\restriction_+)$ is couniform, i.e., every proper closed $\f\restriction_+$-invariant linear subspace of $D_+(V,\f)$ is superfluous in $(\Lambda (D_+(V,\f),\f\restriction_+), +,\cap)$.
\end{itemize}
\begin{rem}\label{rem:alg rk}
By analogy with the topological case, one can define the {\it rank} of an algebraic flow, say $\mathrm{rk}(W,\psi)$, by using the notion of  Goldie dimension: this rank does value the size of the factor of completely positive algebraic entropy, i.e., $\mathrm{rk}(W,\psi)=\dim(W,\psi):=\dim(\Lambda(F_+(W,\psi),\overline{\psi}^+))$, where $\dim$ denotes the Goldie dimension. Since $F_+(W,\psi)$ is  the torsion-free factor of the algebraic flow with respect to the torsion theory $(\ca T_{alg},\ca F_{alg})$, the notion of rank so introduced simply coincides with the torsion-free rank of the algebraic flow $(W,\psi)$ regarded as a module over $\K[t]$ (see Remark~\ref{rem:alg tor}) and so it does not yield any new information.
 \end{rem}
Finally, we are ready to prove that the topological entropy $\ent^*$ is essentially a dimension of the topological flow.
\begin{thmC}\label{thm:main}
For every topological flow $(V,\phi)$, one has
$$\ent^*(V,\phi)=\mathrm{cork}(V,\f).$$
\end{thmC}
\begin{proof}[Proof of Theorem C]  
By Corollary \ref{cor:ent dom}, $V$ can be replaced by its domain of completely positive topological entropy without loss of generality.
So let assume $\ent^*(V,\f)>>0$.

When $\ent^*(V,\f)=0$ there is nothing to prove. 
Thus firstly suppose $\ent^*(V,\f)=k\in\N_{>0}$. By Theorem~B and Proposition~\ref{prop:number}, one obtains $k\leq\mathrm{cork}(V,\f)\leq k$, i.e., $k=\mathrm{cork}(V,\f)$. 

Now assume $\ent^*(V,\f)=\infty$. By Theorem B, $\ent^*(V,\f)\leq\mathrm{cork}(V,\f)$ i.e.,  $\mathrm{cork}(V,\f)=\infty$. 
Vice-versa, let $\mathrm{cork}(V,\f)=\infty$.   Suppose ad absurdum that $\ent^*(V,\f)=k$ is finite. Then there exist $k$ coindependent non-stationary cocyclic $\f$-cotrajectories $\{C_i\mid i=1,\dots,k\}$ by Theorem~B. By Proposition~\ref{prop:cotraj coind}, $\ent^*(V/C_i,\overline\f^i)=1$ for all $i=1,\dots,k$, where $\overline\f^i\colon V/C_i\to V/C_i$ is induced by $\f$. Since the topological entropy and the corank coincide in the finite case, one has $\mathrm{cork}(V/C_i,\overline\f^i)=1$ for all $i=1,\dots,k$ by Proposition~\ref{prop:cotraj coind}. Therefore the lattice $\Lambda(V/C_i,\overline\f^i)$ is couniform for every $i=1,\dots,k$ (see property (C2) above), i.e., $C_i$ is a couniform element of the lattice $\Lambda(V,\f)$. Finally, we claim that $C=C_1\cap\ldots\cap C_k$ is superfluos in the lattice $ \Lambda(V,\f)$, i.e., $C+W\neq V$ for every proper closed $\f$-invariant linear subspace of $V$. Thus Proposition~\ref{prop:goldie}(2) yields the contradiction: $\mathrm{cork}(V,\f)=\codi(V,\f)<\infty$.

Let us prove the claim. Let $W$ be a closed $\f$-invariant $\K$-subspace of $V$ such that $C+W=V$. Then $\{C_1,\dots,C_k,W\}$ is a coindependent set of closed $\f$-invariant $\K$-subspaces of $V$ of cardinality $>\ent^*(V,\f)$. By Proposition~\ref{prop:number} and $\ent^*(V,\f)>>0$, one has $W=C_j$ for some $j\in\{1,\ldots,k\}$, and so $W=C+W=V$.\end{proof}
We conclude  applying Theorem~C to produce an alternative proof of the following significant result. 
  \begin{bridge} Let $(W,\psi)$ be an algebraic flow. Then
 $$\ent(W,\psi)=\ent^*(\widehat W,\widehat\psi).$$
 \end{bridge}
 \begin{proof} 
 By \cite[Corollary 4.8]{GBSalce},
 \begin{equation}\label{eq:1}
 \ent(W,\psi)=\dim(\Lambda(F_+(W,\psi),\overline{\psi}^+))=\rk(W,\psi);
 \end{equation}
 see Remark~\ref{rem:alg rk}. Therefore, one computes
 \begin{align*}
 \ent(W,\psi)&=\dim\Lambda(F_+(W,\psi),\overline{\psi}^+)\quad\text{(by \eqref{eq:1})}\\
 &=\codi{\reallywidehat{\Lambda(F_+(W,\psi),\overline{\psi}^+)}}\quad\text{(by Lefschetz duality for lattices)}\\
 &=\codi{\Lambda(D_+(\widehat W,\widehat\psi),(\widehat\psi)\restriction_+)}\quad\text{(by Theorem~B)}\\
 &=\mathrm{cork}((\widehat W,\widehat\psi))=\ent^*(\widehat W,\widehat\psi)\quad\text{(by definition and Theorem~C).}
 \end{align*}
  \end{proof}
 By using the reductions steps described in \cite{CGB1,CGB2}, one then obtains also the Bridge Theorem  \cite[Theorem 1.3]{CGB2} for locally linearly compact vector spaces as a corollary of the one stated here.
 \section*{Acknowledgements} 
 This work was supported by  EPSRC Grant N007328/1 Soluble Groups and Cohomology and partially supported by
 Programma SIR 2014 by MIUR  (Project GADYGR, Number RBSI14V2LI, cup G22I15000160008).

\medskip

I thank my friends Dikran Dikranjan, Anna Giordano Bruno and Simone Virili for useful discussions on the topic of this paper. Moreover, I thank Anna Giordano Bruno and Peter Kropholler for comments on  early versions of this work. Finally, I thank the referee for helpful suggestions.


\end{document}